\documentclass[12pt, reqno]{amsart}
\usepackage{amssymb,dsfont}
\usepackage{eucal}
\usepackage{amsmath}
\usepackage{amscd}
\usepackage[dvips]{color}
\usepackage{multicol}
\usepackage[all]{xy}           
\usepackage{graphicx}
\usepackage{color}
\usepackage{colordvi}
\usepackage{xspace}
\usepackage{txfonts}
\usepackage{lscape}
\usepackage{tikz} 
\usepackage{bm}

\numberwithin{equation}{section}
\usepackage[shortlabels]{enumitem}
\usepackage{ifpdf}
\ifpdf
 \usepackage[colorlinks,final,backref=page,hyperindex]{hyperref}
\else
\fi
\usepackage{tikz}
\usepackage[active]{srcltx}

\makeatletter

\newtheorem{theorem}{Theorem}[section]
\newtheorem{definition}[theorem]{Definition}
\newtheorem{lemma}[theorem]{Lemma}
\newtheorem{proposition}[theorem]{Proposition}
\newtheorem{remark}[theorem]{Remark}
\newtheorem{corollary}[theorem]{Corollary}

\numberwithin{equation}{section}

\def\({\left(}
\def\){\right)}
\def \<{\langle}
\def \>{\rangle}
\def\al{\alpha}
\def\be{\beta}
\def\dt{\delta}

\def\vf{\varphi}
\def\lmd{\lambda}
\def\g{\mathfrak{g}}
\def\h{\mathfrak{h}}
\def\U{\mathcal{U}}
\def\V{\mathcal V}
\newcommand{\C}{\mathbb C}
\newcommand{\R}{\mathbb{R}}
\newcommand{\N}{\mathbb{N}}
\newcommand{\Z}{\mathbb{Z}}

\newcommand{\spanc}[1]{\mathrm{Span}_{\C}\left\{#1\right\}}
\newcommand{\ov}[1]{\overline{#1}}
\renewcommand{\H}{\mathcal{H}}

\def\cc{\mathfrak c}
\def\vvf{v_{\vf}}
\def\jc{J^C}
\def\hj{\H_J}
\def\gjc{\g_{\jc}}
\def\mgvf{M_{\g,\vf}}
\def\mgjcpsi{M_{\gjc,\psi}}
\def\mgjcpsie{\(M_{\gjc,\psi}\)^e}
\def\mvphi{M_{\V,\phi}}
\def\mhjphi{M_{\hj,\phi}}
\def\mhjphig{\(M_{\hj,\phi}\)^{\g}}
\def\cf{\text{col}(F)}
\def\rf{\text{row}(F)}

\def\invp{\theta^+_{\al,\be}}
\def\invm{\theta^-_{\al,\be}}
\def\coefb{B_{k,j}^{(n)}}

\newcommand{\elei}[2]{I_{#1}^{(#2)}}
\newcommand{\vct}[2]{v_{#1+\frac{#2}p}}
\newcommand{\thl}[1]{\theta(L_{#1})}
\newcommand{\thi}[2]{\theta(I_{#1}^{(#2)})}

\begin{document}
\title[Unitary modules over the gap-$p$ Virasoro algebras]
{Unitary modules over the gap-$p$ Virasoro algebras}
  \author{Chengkang Xu}
  \address{C. Xu: School of Mathematical Sciences, Shangrao Normal Collage, Shangrao, Jiangxi, P. R. China\\ Email: xiaoxiongxu@126.com}

\date{}
\maketitle

\begin{abstract}
For any irreducible Harish-Chandra module $V$ over the gap-$p$ Virasoro algebra, we determine the condition for $V$ to be unitary.
\end{abstract}

\section{Introduction}

Unitary modules for Lie algebras play significant roles in many areas of mathematics and physics,
such as statistical mechanics and the two-dimensional conformal quantum field theory.
In \cite{CP1988}, the authors classified irreducible unitary Harish-Chandra modules for the Virasoro algebra and the Virasoro superalgebras (both Ramond and Neveu-Schwarz cases).
These modules are the modules of intermediate series, highest weight modules or lowest weight modules satisfying certain conditions.
Unitary modules were also considered for many other infinite dimensional Lie algebras over the past decades,
such as the affine Kac-Moody algebras\cite{CP1986,CP1987},
Lie algebras of infinite matrices\cite{Palev1,Palev2},
the twisted Heisenberg-Virasoro algebra\cite{Kwon,ZTL},
the $W$-infinity algebras\cite{KL},
the $W$-algebra $W(2,2)$\cite{ZT1},
the Schr{\"o}dinger-Virasoro algebra\cite{ZT2},
and so on.

In this paper we study unitary modules over the gap-$p$ Virasoro algebra $\g$,
which was first introduced in \cite{Xu}.
The gap-$p$ Virasoro algebra has a close relation to the twisted Heisenberg-Virasoro algebra and the algebras of derivations over rational quantum tori.
It can be realized as the universal central extensions of some subalgebras of these two Lie algebras.
Moreover, when $p=2$, $\g$ is isomorphic to the mirror Heisenberg-Virasoro algebra introduced in \cite{Ba} and named in \cite{LPXZ}.
In \cite{Xu}, irreducible Harish-Chandra modules over $\g$ were classified.
Such modules are modules of intermediate series, highest weight modules or lowest weight modules.

Let us now briefly describe the organization of this paper.
In Section 2, we recall the notions of Harish-Chandra modules and unitary modules, and some related results about the Virasoro algebra and the gap-$p$ Virasoro algebra $\g$.
In Section 3, we explicitly determine the irreducibility of the Verma modules, and the structures of irreducible highest weight modules for $\g$.
Section 4 is dedicated to determine the conjugate-linear anti-involutions of $\g$,
and Section 5 is to determine the conditions for irreducible Harish-Chandra modules over $\g$ to be unitary.

Throughout this paper, we use the symbols $\Z$, $\N$, $\Z_+$, $\R$, and $\C$ to denote the sets of integers, non-negative integers, positive integers, real numbers, complex numbers respectively,
and for any subset $N$ of $\C$ we write $N^\times=N\setminus\{0\}$.
Denote by $S^1$ the set of complex numbers of modulus one.
All vector spaces and algebras are over $\C$, the dual space of a vector space $V$ is denoted by $V^*$, and the universal enveloping algebra of a Lie algebra $\mathcal L$ is denoted by $\U(\mathcal L)$.
We fix a positive integer $p>1$ for this paper,
and for any $k\in\Z$, denote by $\ov k$ the residue of $k$ by $p$.
Finally, for $a\in\C$ we denote by $\ov a$ the conjugate of $a$,
and for any finite set $A$ we denote by $|A|$ the cardinal number of $A$.

\section{Preliminaries}
In this section, we recall some related notations and known results about the Virasoro algebra and the gap-$p$ Virasoro algebra.

\begin{definition}
The \textbf{gap-$p$ Virasoro algebra} is a complex Lie algebra $\g$ with a basis
$$\left\{L_{n},\elei ni, C_j\, \mid \, n\in\Z, 1\leq i\leq p-1, 0\leq j\leq [\frac p2] \right\}$$
and nontrivial relations
\begin{align*}
  &[L_m,L_n]=(m-n)L_{m+n}+\frac1{12}(m^3-m)C_0\dt_{m+n,0},\\
  &[\elei mi,\elei nj]=\left(m+\frac ip\right)\dt_{i+j,p}\dt_{m+n+1,0} C_{\min\{i,p-i\}},\quad
  [L_m,\elei ni]=-\left(n+\frac ip\right)\elei{m+n}i,
\end{align*}
where $m,n\in\Z$, $1\leq i,j\leq p-1$, and
$[\frac p2]$ denotes the largest integer less than or equal to $\frac p2$.
\end{definition}

\begin{remark}
 To simplify notations, we write $C_i=C_{p-i}$ for $[\frac p2]\le i\le p-1$ in this paper.
\end{remark}

Let $\mathcal L$ be a Lie algebra with a triangular decomposition in the sense of \cite{MP}, $\mathcal L=\mathcal L_-\oplus\mathcal L_0\oplus\mathcal L_+$,
where $\mathcal L_0$ is an abelian subalgebra.

\begin{definition}
A module $V$ over $\mathcal L$ is called a \textbf{weight module}
if $V=\oplus_{\lmd\in\mathcal L_0^*}V_\lmd$ where $V_\lmd=\{v\in V\mid xv=\lmd(x)v\text{ for any }x\in\mathcal L_0\}$.
A weight $\mathcal L$-module $V$ is called a \textbf{Harish-Chandra module} if $\dim V_\lmd<\infty$ for any $\lmd\in\mathcal L_0^*$,
called a \textbf{module of intermediate series} if all $\dim V_\lmd\le 1$,
and called a \textbf{highest} (resp. \textbf{lowest}) \textbf{weight module} of \textbf{highest} (resp. \textbf{lowest}) \textbf{weight} $\vf\in\mathcal L_0^*$
if $V$ is generated by a \textbf{highest} (resp. \textbf{lowest}) \textbf{weight vector} $v$ such that $\mathcal L_+v=0$ (resp. $\mathcal L_-v=0$) and $xv=\vf(x)v$ for any $x\in\mathcal L_0$.
\end{definition}

Let $\vf\in\mathcal L_0^*$ and $\C\vvf$ be the 1-dimensional module over $\mathcal L_0\oplus\mathcal L_+$ defined by
\[x\vvf=\vf(x)\vvf,\ \mathcal L_+\vvf=0\quad\text{for any }x\in\mathcal L_0.\]
Then induced $\mathcal L$-module
\[M_{\mathcal L,\vf}=\U(\mathcal L)\otimes_{\U(\mathcal L_0\oplus\mathcal L_+)}\C\vvf\]
is called the \textbf{Verma module} of highest weight $\vf$.
For any highest weight module $V$ of highest weight $\vf$ with a highest weight vector $v$,
there is a natural $\mathcal L$-module epimorphism from $M_{\mathcal L,\vf}$ to $V$ defined by $\vvf\mapsto v$.

Set
\[\h=\spanc{L_0,C_i\mid 0\le i\le p-1},\]
which is a Cartan subalgebra of $\g$.
Clearly, $\h$ exhausts all semisimple elements in $\g$,
and the gap-$p$ Virasoro algebra $\g$ is equipped with a triangular decomposition $\g=\g_-\oplus\h\oplus\g_+$, where
\[\g_-=\spanc{L_n,\elei ni\mid n<0,1\le i\le p-1},\quad
  \g_+=\spanc{L_{n+1},\elei ni\mid n\in\N,1\le i\le p-1}.\]
Set
\[\V=\spanc{L_n,C_0\mid n\in\Z}
\quad\text{and}\quad
 \H=\spanc{\elei ni, C_i\mid n\in\Z,1\le i\le p-1}.\]
Clearly, $\V$ is a Virasoro algebra, $\H$ is a Heisenberg algebra and an ideal of $\g$.

Now we recall irreducible Harish-Chandra modules over the Virasoro algebra $\V$ and the gap-$p$ Virasoro algebra $\g$,
which were classified in \cite{Ma} and \cite{Xu} respectively.
They are modules of intermediate series, highest weight modules or lowest weight modules.

The Virasoro algebra $\V$ has a triangular decomposition $\V=\V_-\oplus\V_0\oplus\V_+$, where
\[\V_{\pm}=\V\cap\g_\pm=\spanc{L_n\mid \pm n\in\Z_+},\quad
   \V_0=\V\cap\h=\spanc{L_0,C_0}.\]
Then we have the notion of Verma module $\mvphi$ over $\V$ of highest weight $\phi\in\V_0^*$.
The following results are well known, and one may see \cite{FF} for reference.
\begin{proposition}\label{irrVmod}
 The Verma module $\mvphi$ over $\V$ is irreducible if and only if $\Phi_{h,c}(\al,\be)=0$ for some $\al,\be\in\Z_+$, where $h=\phi(L_0), c=\phi(C_0)$ and
 \[\Phi_{h,c}(\al,\be)=\(h+\frac{(\al^2-1)(c-13)}{24}+\frac{(\al\be-1)}2\)
                \(h+\frac{(\be^2-1)(c-13)}{24}+\frac{(\al\be-1)}2\)+\frac{(\al^2-\be^2)^2}{16}.\]
\end{proposition}

For any $a,b\in\C$, there is a $\V$-module structure on the space $\oplus_{n\in\Z}\C v_n$ with actions
\[L_nv_k=-(a+k+bn)v_{n+k},\quad C_0v_k=0\quad\text{for any }n,k\in\Z.\]
The resulting $\V$-module of intermediate series is usually denoted by $V(a,b)$.
It is well known that $V(a,b)$ is reducible if and only if $a\in\Z$ and $b\in\{0,1\}$.

Let $a,b\in\C$ and $F=(F_{i,j})$ be a $(p-1)\times p$ matrix with index $1\le i\le p-1, 0\le j\le p-1$ satisfying that
\[ F_{s,i}F_{r,\,\ov{i+s}}= F_{r,i} F_{s,\,\ov{i+r}} \quad\text{for any }1\leq r,s\le p-1, 0\le i\le p-1, \]
and
\[i+j\in\cf\quad \text{if } j\in\cf \text{ and } i\in\rf,\]
where $\cf=\{j\mid F_{ij}\neq0\text{ for some }i\}$ and $\rf=\{i\mid F_{ij}\neq0\text{ for some }j\}$.
Then there is a $\g$-module structure on
the space $\oplus_{j\in\cf,k\in\Z}\C\vct kj$ with actions
\[ L_m\vct kj=-\(a+k+\frac jp+bm\)\vct{m+k}j,\quad
 \elei mi\vct kj=F_{i,j}\vct{m+k}{i+j}, \quad C_s\vct kj=0\]
for $m,k\in\Z, j\in\cf, 1\le i\le p-1, 0\le s\le p-1$.
The resulting $\g$-module is a module of intermediate series and we denote it by $V(a,b,F)$ in this paper.
For any $j\in\cf$, denote
$$V_{(j)}=\sum_{k\in\Z}\C v_{k+\frac jp},$$
which is a $\V$-module of intermediate series isomorphic to $V(a+\frac jp,b)$,
and $V(a,b,F)=\bigoplus_{j\in\cf}V_{(j)}$.
\begin{proposition}\cite[Proposition 4.2]{Xu}
 The $\g$-module $V(a,b,F)$ is reducible if and only if $|\cf|=1, a\in\Z$ and $b\in\{0,1\}$.
\end{proposition}

Next we recall some notions and results about unitary modules over $\V$.

\begin{definition}
 Let $\mathcal L$ be a Lie algebra.
 An $\mathcal L$-module $V$ is called \textbf{unitary} if $V$ admits a positive definite Hermitian form $\<\cdot,\cdot\>$ such that
 \[\<xu,v\>=\<u,\theta(x)v\>\quad\text{for any }x\in\mathcal L,u,v\in V,\]
 where $\theta:\mathcal L\longrightarrow\mathcal L$ is a \textbf{conjugate-linear anti-involution} on $\mathcal L$, i.e., a map satisfying that
 \[\theta^2=\mathrm{identity},\ \theta(x+y)=\theta(x)+\theta(y),\ \theta(ax)=\ov a\theta(x),\ \theta([x,y])=[\theta(y),\theta(x)]\]
 for any $a\in\C, x,y\in\mathcal L$.
\end{definition}

It is worthwhile to notice that unitary weight modules are completely reducible.
The following propositions are from \cite{CP1988}.
\begin{proposition}\label{prop:invvir}
 Any conjugate-linear anti-involution of $\V$ is one of the following types.\\
 (i) $\theta_\al^+(L_n)=\al^nL_{-n},\ \theta_\al^+(C_0)=C_0$ with $\al\in\R^\times$.\\
 (ii) $\theta_\al^-(L_n)=-\al^nL_{n},\ \theta_\al^-(C_0)=-C_0$ with $\al\in S^1$ the set of complex numbers of modulus one.
\end{proposition}

\begin{proposition}\label{prop:unitarycon}
 Let $V$ be a nontrivial irreducible weight module over $\V$.\\
 (1) If $V$ is unitary for some conjugate-linear anti-involution $\theta$ of $\V$, then $\theta=\theta_\al^+$ for some $\al>0$.\\
 (2) If $V$ is unitary for $\theta_\al^+$ with $\al>0$, then $V$ is unitary for $\theta_1^+$.
\end{proposition}

\begin{proposition}\label{prop:unitaryVmod}
 (1) The $\V$-module $V(a,b)$ is unitary for $\theta_1^+$ if and only if $a\in\R,b\in\frac12+\sqrt{-1}\R$.\\
 (2) The irreducible highest weight module of highest weight $\phi$ over $\V$ is unitary for $\theta_1^+$ if and only if $\phi(C_0)\ge1,\phi(L_0)\ge0$, or there exists an integer $m\ge2$ and $r,s\in\Z$ such that
 \[0\le r<s<m,\quad \phi(C_0)=1-\frac6{m(m+1)},\quad \phi(L_0)=\frac{(mr+s)^2-1}{4m(m+1)}.\]
\end{proposition}

\section{Highest weight modules}\label{sec3}

In this section we determine the irreducibility of the Verma modules and structures of irreducible highest weight modules over $\g$.

Recall the Verma $\g$-module $\mgvf$ of highest weight $\vf\in\h^*$.
Denote
\[J=\{1\le i\le p-1\mid \vf(C_i)\neq0\}\quad\text{and}\quad
    \jc=\{1\le i\le p-1\mid \vf(C_i)=0\}.\]
Since $C_i=C_{p-i}$, the set $J$ is \textbf{symmetrical} in the sense that $p-i\in J$ if $i\in J$. The set $\jc$ is also symmetrical.
Then we have subalgebras
$$\gjc=\spanc{L_n,\elei ni, C_0,C_i\mid n\in\Z,i\in\jc}\text{ and }
  \hj=\spanc{\elei ni, C_i\mid n\in\Z,i\in J}$$
of $\g$.
Clearly, $\hj$ is an ideal of $\g, \gjc\ltimes\hj=\g$,
and $\hj=\H,\gjc=\V$ if $\jc=\emptyset$.
Moreover, they have triangular decompositions induced from $\g$.
Define a linear function $\psi\in\(\h\cap\gjc\)^*$ by
\begin{equation}\label{psi}
  \psi(L_0)=\vf(L_0)-\sum_{j\in J}\frac{j(p-j)}{4p^2};\quad
  \psi(C_0)=\vf(C_0)-|J|; \quad \psi(C_i)=0\quad\text{ for }i\in\jc.
\end{equation}
Then we have the Verma module $\mgjcpsi$ of $\gjc$ with highest weight $\psi$.
Clearly, if $\jc\neq \emptyset$,
then $\U\(\H_{\{i,p-i\}}\)v_\psi$ generates a proper $\g$-submodule of $\mgjcpsi$,
where $i\in\jc$,
\[\H_{\{i,p-i\}}=\bigoplus_{n\in\Z}\C\elei ni\oplus\C\elei n{p-i},\]
and $v_\psi$ is a highest weight vector in $\mgjcpsi$.
So $\mgjcpsi$ is reducible.
Denote by $\psi\mid_\V$ the restriction of $\psi$ to $\h\cap\V$.
Since any highest weight $\V$-module of highest weight $\psi\mid_\V$ can be extended to a highest weight $\gjc$-module of highest weight $\psi$ by requiring that $\elei ni$ and $C_i$ act trivially for all $n\in\Z$ and $i\in\jc$,
we see that the irreducible quotient of $\mgjcpsi$ is equivalent to the irreducible quotient of the Verma $\V$-module $M_{\V,\psi\mid_\V}$.

By demanding $\hj\mgjcpsi=0$, we turn $\mgjcpsi$ into a $\g$-module,
and should denote the resulting $\g$-module by $\mgjcpsie$.

Set $\phi=\vf\mid_{\hj}$ the restriction of $\vf$ to $\h\cap\hj$.
We have the Verma module $\mhjphi$ of $\hj$ with highest weight $\phi$,
which one can easily prove is irreducible.
The following result is a special case of \cite[Proposition 3.1]{XCT}.

\begin{proposition}\label{gmodext}
 There is an irreducible highest weight $\g$-module structure on the space $\mhjphi$ with actions given by
 \begin{align}\label{eq:Iaction}
  \elei ni &\mapsto \begin{cases} \elei ni\ &\text{if}\ i\in J,\\
                                        0\ &\text{if}\ i\notin J,\end{cases}\ \ \ \ \quad
                   C_i\mapsto \begin{cases} \phi(C_i)\ &\text{if}\ i\in J,\\
                                   0\ &\text{if}\ i\notin J,\end{cases},\quad C_0\mapsto |J|,\\
 \label{eq:Lnaction}
    L_n   &\mapsto \sum_{j\in J}\frac{1}{2\phi(C_j)}\left(\sum_{k\in \Z} :
                    \elei kj\elei{n-k-1}{p-j}:\right)
                   +\dt_{n,0}\sum_{j\in I}\frac{j(p-j)}{4p^2},
 \end{align}
 where $n\in\Z, 1\leq i\leq p-1$ and the normal ordered product is defined by
\begin{equation*}
   :\elei ki\elei l{p-i}:=\begin{cases}
      \elei ki\elei l{p-i}\quad&\text{ if }k<l;\\
       \elei l{p-i}\elei ki \quad&\text{ if }k\ge l.
   \end{cases}
\end{equation*}
\end{proposition}
We denote by $\mhjphig$ the resulting $\g$-module from Proposition \ref{gmodext}.
Moreover, we have a characterization of the Verma module $\mgvf$ of $\g$.
\begin{theorem}\label{thm:irrgmod}
 (1) $\mgvf\cong\mhjphig\otimes\mgjcpsie$ if $J\neq\emptyset$.\\
 (2) Any submodule of $\mhjphig\otimes\mgjcpsie$ is of the form $\mhjphig\otimes W^e$,
 where $W$ is a $\gjc$-submodule of $\mgjcpsi$ and $W^e$ is the $\g$-module obtained from $W$ in the same fashion as $\mgjcpsie$.
 Then the irreducible highest weight $\g$-module of highest weight $\vf$ is isomorphic to $\mhjphig\otimes \(V_{\V,\psi\mid_\V}\)^e$ if $J\neq\emptyset$,
 and isomorphic to $\(V_{\V,\psi\mid_\V}\)^e$ if $J=\emptyset$,
 where $V_{\V,\psi\mid_\V}$ is the unique irreducible quotient of the Verma $\V$-module $M_{\V,\psi\mid_\V}$
 (i.e., the irreducible highest weight $\V$-module of highest weight $\psi\mid_\V$),
 and $\(V_{\V,\psi\mid_\V}\)^e$ is the $\g$-module obtained from $V_{\V,\psi\mid_\V}$ in the same fashion as $\mgjcpsie$.\\
 (3) The Verma module $\mgvf$ is irreducible if and only if $J=\{1,2,\dots,p-1\}$ and
 \[\Phi(\al,\be)\Phi(\be,\al)+(\al^2-\be^2)^2=0\text{ for some }
  \al,\be\in\Z_+,\]
 where
 \[\Phi(\al,\be)=4\vf(L_0)-\sum_{j=1}^{p-1}\frac{j(p-j)}{p^2}+\frac{(\al^2-1)(\vf(C_0)-p-12)}6
     +2(\al\be-1).\]
\end{theorem}
\begin{proof}
 One can check that $\mhjphig\otimes\mgjcpsie$ is a highest weight $\g$-module of highest weight $\vf$.
 Then (1) follows from the PBW theorem.
 (2) follows from (1) and \cite[Proposition 3.5]{XCT},
 and (3) follows from (2) and Proposition \ref{irrVmod}.
\end{proof}

\section{Conjugate-linear anti-involutions of $\g$}

In this section we determine all conjugate-linear anti-involutions of $\g$.
Denote 
$$\cc=\spanc{C_i\mid 0\le i\le p-1}$$
the center of $\g$.
Then $\h=\C L_0\oplus\cc$.

\begin{lemma}\label{lem:invg}
 Let $\theta$ be a conjugate-linear anti-involution of $\g$. Then\\
 (1) $\theta(\cc)=\cc,\theta(\h)=\h$ and $\theta(L_0)=\lmd L_0+C(L_0)$, where $\lmd=\pm1$ and $C(L_0)\in\cc$.\\
 (2) $\theta(C_i)=\lmd  A_{0,i} A_{-1,p-i}C_i$ and
 \begin{equation}\label{value:ini}
   \theta(\elei ni)=\begin{cases}
       A_{-n-1,p-i}\elei{-n-1}{p-i} &\text{if }\lmd=1,\\
       A_{n,i}\elei ni &\text{if }\lmd=-1,
   \end{cases}
 \end{equation}
 where $n\in\Z, 1\le i\le p-1$ and all $ A_{j,k}$ are nonzero complex numbers.
 In particular, $\theta(\H)=\H$.
\end{lemma}
\begin{proof}
 (1) Let $x\in\cc, y\in\g$ and $z=\theta(y)$.
 Then $y=\theta^2(y)=\theta(z)$ and $[\theta(x),y]=\theta[z,x]=0$, proving that $\theta(\cc)=\cc$.
 Since
 \[[\theta(L_0),\theta(L_n)]=\theta[L_n,L_0]=n\theta(L_n)\quad\text{and}\quad
   [\theta(L_0),\theta(\elei ni)]=\theta[\elei ni,L_0]=\(n+\frac ip\)\theta(\elei ni)\]
 for any $n\in\Z,1\le i\le p-1$, we see that $\theta(L_0)\in\h\setminus\cc$.
 So $\theta(\h)=\h$ and we may write
 \[\theta(L_0)=\lmd L_0+C(L_0)\quad\text{for some } \lmd\in\C, C(L_0)\in\cc.\]
 From $L_0=\theta^2(L_0)=\theta(\lmd L_0+C(L_0))=\ov\lmd\lmd L_0+\ov\lmd C(L_0)+\theta(C(L_0))$ we know that $\lmd\in S^1$.

 For any $n\in\Z,1\le i\le p-1$, write (a finite sum)
 \[\theta(\elei ni)=\sum_{k\in\Z}\sum_{j=1}^{p-1} A_{k,j}\elei kj+C(\elei ni),
  \quad\text{where }C(\elei ni)\in\cc,  A_{k,j}\in\C^\times.\]
 From the equation
 \begin{equation}\label{eq:lmd}
  \(n+\frac ip\)\theta(\elei ni)=\theta[\elei ni,L_0]=[\theta(L_0),\theta(\elei ni)]
   =-\lmd\sum_{k\in\Z}\sum_{j=1}^{p-1}\(k+\frac jp\) A_{k,j}\elei kj
 \end{equation}
 we see that $-\lmd\(k+\frac jp\)=n+\frac ip$ for a fixed pair of $k,j$.
 This implies that $\lmd$ is a rational number and hence $\lmd=\pm1$, proving (1).
 Moreover, we have
 \[\begin{cases} k=-n-1, j=p-i & \text{ if }\lmd=1,\\
              k=n, j=i & \text{ if }\lmd=-1.
   \end{cases}\]

 (2) The equation \eqref{eq:lmd} also implies that $C(\elei ni)=0$.
 Then we get \eqref{value:ini} from the above equation.
 For any $1\le i\le p-1$, one can see from the following equation
 \[\frac ip\theta(C_i)=\theta[\elei 0i,\elei{-1}{p-i}]=[\theta(\elei {-1}{p-i}),\theta(\elei 0i)]=\lmd A_{0,i} A_{-1,p-i}\frac ip C_i\]
 that $\theta(C_i)=\lmd  A_{0,i} A_{-1,p-i}C_i$.
 This completes the proof.
\end{proof}

\begin{proposition}\label{prop:invg}
 Any conjugate-linear anti-involution of $\g$ is one of the following types.\\
 (1) $\invp(L_n)=\al^nL_{-n},\ \invp(\elei ni)=\al^{n}\be_{p-i}\elei{-n-1}{p-i},\
     \invp(C_0)=C_0,\ \invp(C_i)=\al^{-1}\be_i\,\be_{p-i}C_i$
     with $\al\in\R^\times , \be_i\in\C^\times $ such that $\ov{\be_i}\,\be_{p-i}=\al$ for all $1\le i\le p-1$, and $\be=(\be_1,\dots,\be_{p-1})$;\\
 (2) $\invm(L_n)=-\al^nL_{n},\ \invm(\elei ni)=\al^n\be_i\elei ni,\
     \invm(C_0)=-C_0,\ \invm(C_i)=-\al^{-1}\be_i\be_{p-i}C_i$
      with $\al, \be_i\in S^1$.
\end{proposition}
\begin{proof}
 Let $\theta$ be a conjugate-linear anti-involution of $\g$,
 which fixes $\H$ by Lemma \ref{lem:invg}(2) and hence induces a conjugate-linear anti-involution, denoted by $\ov\theta$, on the quotient algebra $\g/\H$.
 For an element $x\in\g$, we denote by $\ov x$ the image of $x$ in $\g/\H$ in the following.
 Since $\g/\H\cong\V$, we see from Proposition \ref{prop:invvir} that $\ov\theta$ is one of the following two types.\\
 (i) $\ov{\theta_\al^+}: \ov{L_n}\mapsto\al^n\ov{L_{-n}},\ \ov{C_0}\mapsto\ov{C_0}$ for some $\al\in\R^\times $.\\
 (ii) $\ov{\theta_\al^-}: \ov{L_n}\mapsto-\al^n\ov{L_n},\ \ov{C_0}\mapsto-\ov{C_0}$ for some $\al\in S^1$.\\
 In the following we divide the proof into two cases according to the two types of $\ov\theta$.

 \textbf{Case 1}: $\ov\theta=\ov{\theta_\al^+}$.
 For any $n\in\Z$ we may write (a finite sum)
 \[\theta(L_n)=\al^nL_{-n}+\sum_{k\in\Z}\sum_{j=1}^{p-1}\coefb\elei kj+C(L_n),\]
 where $\coefb\in\C, C(L_n)\in\cc$.
 In particular, comparing with Lemma \ref{lem:invg} we see that $\lmd=1$ and $\theta(L_0)=L_0+C(L_0)$.
 Moreover, we have $\thi ni=A_{-n-1,p-i}\elei{-n-1}{p-i}$ by \eqref{value:ini}.

 For $n\in\Z^\times $ we have
 \[n\thl n=[\thl0,\thl n]
   =[L_0,\al^nL_{-n}+\sum_{k\in\Z}\sum_{j=1}^{p-1}\coefb\elei kj]
   =n\al^nL_{-n}-\sum_{k\in\Z}\sum_{j=1}^{p-1}\(k+\frac jp\)\coefb\elei kj.
 \]
 This gives that $C(L_n)=0$ and $\coefb=0$ for all $k,j$.
 So $\thl n=\al^nL_{-n}$ for $n\neq0$.
 From
 \begin{align*}
   &\theta[L_n,L_{-n}]=2n\thl0+\frac1{12}(n^3-n)\theta(C_0)
    =2nL_0+2nC(L_0)+\frac1{12}(n^3-n)\theta(C_0)\\
    &=[\thl{-n},\thl n]=2nL_0+\frac1{12}(n^3-n)C_0
 \end{align*}
 we see that $C(L_0)=0$ and $\theta(C_0)=C_0$.
 Therefore $\thl n=\al^nL_{-n}$ for any $n\in\Z$.

 Using \eqref{value:ini} and expanding $\theta[L_n,\elei 0i]=[\thi0i,\thl n]$ for any $n\in\Z,1\le i\le p-1$,
 we get $ A_{-n-1,p-i}=\al^n A_{-1,p-i}$, or equivalently,
 \[ A_{n,i}=\al^{-n} A_{0,i}\quad\text{for any}\ n\in\Z,1\le i\le p-1.\]
 Set $\be_i=\al A_{0,i}$ for any $1\le i\le p-1$ and set $\be=(\be_1,\be_2,\dots,\be_{p-1})$.
 Then we have $\theta(\elei ni)=\al^n\be_{p-i}\elei {-n-1}{p-i}$
 and $\theta(C_i)=\al^{-1}\be_i\be_{p-i}C_i$ by Lemma \ref{lem:invg}(2).
 Moreover, from $\elei ni=\theta^2(\elei ni)=\al^{-1}\ov{\be_{p-i}}\,\be_i\elei ni$ we get $\ov{\be_{p-i}}\,\be_i=\al=\ov{\be_i}\,\be_{p-i}$.
 This shows $\theta=\invp$.

 \textbf{Case 2}: $\ov\theta=\ov{\theta_\al^-}$.
 Similarly this case leads to $\theta=\invm$.
\end{proof}

\begin{corollary}
 The subalgebras $\V,\H,\gjc,\hj$ are stable under any conjugate-linear anti-involution of $\g$.
\end{corollary}
\begin{proof}
 It follows from Proposition \ref{prop:invg}.
\end{proof}

\section{Unitary Harish-Chandra modules}

In this section we determine the condition for any irreducible Harish-Chandra module over $\g$ to be unitary.
First we have the following
\begin{proposition}\label{prop:uninv}
Let $V$ be a nontrivial irreducible weight $\g$-module which is unitary for some conjugate-linear anti-involution of $\g$, then $V$ is unitary for some $\theta_{1,\be}^+$ where $\be=(\be_1,\be_2,\dots,\be_{p-1})\in \C^{p-1}$ such that $\ov{\be_i}\,\be_{p-i}=1$ for all $1\le i\le p-1$.
\end{proposition}
\begin{proof}
 Suppose $V$ is unitary for a conjugate-linear anti-involution $\theta$ of $\g$.
 Consider $V$ as a $\V$-module and $V$ is unitary for $\theta|_\V$ the restriction of $\theta$ to $\V$.
 We claim that $V$ is nontrivial as a $\V$-module.
 Otherwise, for any $v\in V, n\in\Z,1\le i\le p-1$,
 it follows from $[L_0,\elei ni]v=0$ that $\elei niv=0$.
 Hence $C_iv=0$ and $V$ is a trivial $\g$-module, a contradiction.

 Since unitary weight $\V$-modules are completely reducible,
 $V$ contains a nontrivial irreducible unitary $V$-submodule for $\theta|_\V$.
 So we may assume $\theta|_\V=\theta_1^+$ by Proposition \ref{prop:unitarycon}.
 Then the statement follows from Proposition \ref{prop:invg}.
\end{proof}

In the following we determine the condition for irreducible Harish-Chandra $\g$-modules to be unitary for a fixed conjugate-linear anti-involution $\theta_{1,\be}^+$ of $\g$ as in Proposition \ref{prop:uninv}.
We first do it for the $\g$-module $V(a,b,F)$ of intermediate series.
Recall the set
$$\cf=\{0\le j\le p-1\mid F_{ij}\neq 0\ \text{ for some }i\}.$$

\begin{theorem}\label{thm:unitaryintser}
 The $\g$-module $V(a,b,F)$ of intermediate series is unitary for $\theta_{1,\be}^+$
 if and only if $a\in\R, b\in\frac12+\sqrt{-1}\R$ and $\be_{p-i}F_{p-i,\ov{i+j}}=\ov{F_{i,j}}$ for any $j\in\cf,1\le i\le p-1$.
\end{theorem}
\begin{proof}
 If $a\in\R, b\in\frac12+\sqrt{-1}\R$ and $\be_{p-i}F_{p-i,\ov{i+j}}=\ov{F_{i,j}}$ for any $j\in\cf,1\le i\le p-1$,
 then it is easy to check that $V(a,b,F)$ is unitary for $\theta_{1,\be}^+$ and
 the positive definite Hermitian form $\<\cdot,\cdot\>$ defined by
 \[\<v_{m+\frac ip},v_{n+\frac jp}\>=\dt_{m,n}\dt_{i,j}
 \quad\text{for }m,n\in\Z,i,j\in\cf.\]

 Conversely, suppose that $V(a,b,F)$ is unitary for $\theta_{1,\be}^+$ and some positive definite Hermitian form $\<\cdot,\cdot\>$.
 Then all $V_{(j)}$'s are unitary $\V$-modules for $\(\theta_{1,\be}^+\)|_\V=\theta_{1}^+$.
 Therefore, $a\in\R, b\in\frac12+\sqrt{-1}\R$ by Proposition \ref{prop:unitaryVmod}(1).
 Moreover, for any $j\in\cf,1\le i\le p-1$,
 normalize the Hermitian form $\<\cdot,\cdot\>$ such that $\<v_{\frac{i+j}p},v_{\frac{i+j}p}\>=\<v_{\frac jp},v_{\frac jp}\>=1$
 by replacing the basis elements $v_{\frac jp}$ with a suitable multiple if necessary.
 Then we have
 \begin{align*}
 \be_{p-i}F_{p-i,\ov{i+j}}&=\be_{p-i}F_{p-i,\ov{i+j}}\<v_{\frac jp},v_{\frac jp}\>
   =\<\be_{p-i}\elei{-1}{p-i}v_{\frac{i+j}p},v_{\frac jp}\>
   =\<\theta_{1,\be}^+(\elei 0i)v_{\frac{i+j}p},v_{\frac jp}\>\\
   &=\<v_{\frac{i+j}p},\elei 0i v_{\frac jp}\>
   =\ov{F_{i,j}}\,\<v_{\frac{i+j}p},v_{\frac{i+j}p}\>=\ov{F_{i,j}}.
 \end{align*}
 This completes the proof.
\end{proof}

Next we are to determine the condition for irreducible highest weight $\g$-modules to be unitary for $\theta_{1,\be}^+$ and some positive definite Hermitian form.

Recall the irreducible highest weight $\g$-module with highest weight $\vf$,
which by Theorem \ref{thm:irrgmod} is isomorphic to $V_{\V,\psi\mid_\V}$ if $J=\emptyset$,
and isomorphic to $\mhjphig\otimes\(V_{\V,\psi\mid_\V}\)^e$ if $J\neq\emptyset$.
The following lemma is easy to check.

\begin{lemma}\label{lem:tensorunitarymod}
 Let $V,W$ be two unitary module over a Lie algebra $\mathcal L$ for some conjugate-linear anti-involution $\theta$ and positive definite Hermitian forms $\<\cdot,\cdot\>_V$ and $\<\cdot,\cdot\>_W$ separately.
 Then the tensor product $\g$-module $V\otimes W$ is unitary for $\theta$ and the positive definite Hermitian form $\<\cdot,\cdot\>$ defined by
 \[\<v_1\otimes w_1,v_2\otimes w_2\>=\<v_1,v_2\>_V\<w_1,w_2\>_W
    \quad\text{ for }v_1,v_2\in V, w_1,w_2\in W.\]
\end{lemma}

Notice that the $\g$-module $\mhjphig$ is a highest weight module with highest weight $\phi^\g$ such that
\[\phi^\g(L_0)=\sum_{j\in J}\frac{j(p-j)}{4p^2},\quad \phi^\g(C_0)=|J|,\quad
  \phi^\g(C_i)=\begin{cases}
           \phi(C_i)=\vf(C_i) &\text{if }i\in J;\\
                           0  &\text{if }i\notin J.
        \end{cases}
\]
\begin{proposition}\label{prop:unitaryHJmod}
 The $\g$-module $\mhjphig$ is unitary for $\theta_{1,\be}^+$ if and only if $\be_i\phi(C_i)\in\R^\times $ for all $i\in J$.
\end{proposition}
\begin{proof}
 Suppose that $\mhjphig$ is unitary for $\theta_{1,\be}^+$.
 Let $v$ be a nonzero highest weight vector in $\mhjphig$.
 Then for any $i\in J$, we have (noting that $\be_i\,\ov{\be_{p-i}}=1$)
 \begin{align*}
  \frac ip\be_i&\phi(C_i)\<v,v\>=\be_i\<[\elei 0i,\elei{-1}{p-i}]v,v\>
     =\be_i\<\elei 0i\elei{-1}{p-i}v,v\>
     =\be_i\<v,\theta_{1,\be}^+(\elei{-1}{p-i})\theta_{1,\be}^+(\elei 0i) v\>\\
     &=\be_i\,\ov{\be_i\,\be_{p-i}}\<v,\elei 0i\elei{-1}{p-i}v\>
     =\ov{\be_i}\<v,[\elei 0i,\elei{-1}{p-i}]v\>
     =\frac ip\ov{\be_i\phi(C_i)}\<v,v\>,
 \end{align*}
 which implies that $\ov{\be_i\phi(C_i)}={\be_i\phi(C_i)}$.
 So $\be_i\phi(C_i)\in\R^\times $.

 Conversely, suppose $\be_i\phi(C_i)\in\R^\times $ for all $i\in J$.
 Notice that the universal enveloping algebra $\U(\hj)$ has a decomposition
 \[\U(\hj)=\U(\hj\cap\cc)\oplus\(\sum_{i\in J}\sum_{n\in\Z_+}\elei {-n}i\U(\hj)
          +\U(\hj)\elei {n-1}i\).\]
 Denote by $\rho:\U(\hj)\longrightarrow\U(\hj\cap\cc)$ the projection onto the first summand in the above decomposition.
 Then the space $\mhjphi$ admits a positive definite Hermitian form $\<\cdot,\cdot\>$ defined by
 \[\<xv,yv\>v=\rho\(\widetilde{\theta_{1,\be}^+}(x)y\)v,\]
 where $x,y\in\U\(\sum_{i\in J}\sum_{n\in\Z_+}\C\elei {-n}i\)$
 and $\widetilde{\theta_{1,\be}^+}$ is the anti-involution of $\U(\hj)$ extended from
 $\theta_{1,\be}^+$ by
 \[\widetilde{\theta_{1,\be}^+}(x_1\cdots x_k)
    =\theta_{1,\be}^+(x_k)\cdots\theta_{1,\be}^+(x_1)
    \quad\text{ for any } x_1,\dots ,x_k\in\hj.\]
 Then it is clear that the $\hj$-module $\mhjphi$ is unitary for $\theta_{1,\be}^+|_{\hj}$.

 Now for $n\in\Z^\times , u,w\in\mhjphig$,
 we have (noting that $\be_i\,\ov{\be_{p-i}}=1$)
 \begin{align*}
  &\<L_nu,w\>=\sum_{j\in J}\frac1{2\phi(C_j)}\sum_{k\in\Z}
            \<\elei kj\elei{n-k-1}{p-j}u,w\>
    =\sum_{j\in J}\frac1{2\phi(C_j)}\sum_{k\in\Z}\<u,
          \theta_{1,\be}^+\(\elei{n-k-1}{p-j}\)\theta_{1,\be}^+\(\elei kj\) w\>\\
  &=\sum_{j\in J}\frac{\ov{\be_j\be_{p-j}}}{2\phi(C_j)}\sum_{k\in\Z}\<u,
          \elei {-n+k}j\elei{-k-1}{p-j} w\>
   =\sum_{j\in J}\frac{1}{2\ov{\phi(C_j)}}\sum_{k\in\Z}\<u,
          \elei {k}j\elei{-n-k-1}{p-j} w\>\\
   &=\<u,L_{-n}w\>=\<u,\theta_{1,\be}^+(L_n)w\>.
 \end{align*}
 Since $L_0$ acts semisimply on $\mhjphig$, it is clear that $\<L_0u,w\>=\<u,L_0w\>$.
 Moreover, for any $i\in\jc,n\in\Z, u,w\in\mhjphig$,
 we have $\<\elei ni u,w\>=0=\<u,\theta_{1,\be}^+(\elei ni)w\>.$
 So we have
 \[\<xu,w\>=\<u,\theta_{1,\be}^+(x)w\>\quad\text{for any }x\in\g,u,w\in\mhjphig.\]
 This shows that the $\g$-module $\mhjphig$ is unitary for $\theta_{1,\be}^+$.
\end{proof}

\begin{proposition}\label{prop:unitaryVmode}
 The $\g$-module $\(V_{\V,\psi\mid_\V}\)^e$ is unitary for $\theta_{1,\be}^+$ if and only if $\psi(C_0)\ge1,\psi(L_0)\ge0$, or there exists an integer $m\ge2$ and $r,s\in\Z$ such that
 \[0\le r<s<m,\quad \psi(C_0)=1-\frac6{m(m+1)},\quad \psi(L_0)=\frac{(mr+s)^2-1}{4m(m+1)}.\]
\end{proposition}
\begin{proof}
 Since $\H\(V_{\V,\psi\mid_\V}\)^e=0$, it is clear that the $\g$-module $\(V_{\V,\psi\mid_\V}\)^e$ is unitary for $\theta_{1,\be}^+$ if and only if the $\V$-module $V_{\V,\psi\mid_\V}$ is unitary for $\theta_{1,\be}^+|_\V=\theta_{1}^+$.
 Then it follows from Proposition \ref{prop:unitaryVmod}.
\end{proof}

\begin{theorem}\label{thm:unitaryhimod}
 The irreducible highest weight $\g$-module of highest weight $\vf$ is unitary for the conjugate-linear anti-involution $\theta_{1,\be}^+$ if and only if the following two statements stand.\\
 (1) $\be_i\vf(C_i)\in\R^\times $ for all $i\in J$;\\
 (2) $\vf(C_0)\ge|J|+1,\vf(L_0)\ge\sum_{j\in J}\frac{j(p-j)}{4p^2}$, or there exists an integer $m\ge2$ and $r,s\in\Z$ such that
 \[0\le r<s<m,\quad \vf(C_0)=|J|+1-\frac6{m(m+1)},\quad \vf(L_0)=\sum_{j\in J}\frac{j(p-j)}{4p^2}+\frac{(mr+s)^2-1}{4m(m+1)}.\]
\end{theorem}
\begin{proof}
It follows from Lemma \ref{lem:tensorunitarymod}, Proposition \ref{prop:unitaryHJmod} and Proposition \ref{prop:unitaryVmode}.
\end{proof}

Denote by $\omega$ the Chevalley involution of $\g$ defined by
\[\omega(L_n)=-L_{-n},\ \omega(\elei ni)=-\elei{-n}{p-i},
\ \omega(C_j)=-C_j\ \text{for any }n\in\Z, 1\le i\le p-1, 0\le j\le p-1.\]

Let $V$ be an irreducible lowest weight module of $\g$. Then the following action
\[x\cdot w=\omega(x)w\quad\text{for any }x\in \g, w\in V,\]
makes $V$ a highest weight $\g$-module.
We denote the resulting module by $V^*$.
It is easy to check that $V$ is unitary under some positive definite Hermitian form $<\cdot,\cdot>$ and $\theta_{1,\be}^+$ if and only if
$V^*$ is unitary under $<\cdot,\cdot>$ and $\theta_{1,\be'}^+$,
where $\be'=(\be_1',\dots,\be_{p-1}')$ with $\be_i'=\be_{p-i}$ for all $1\le i\le p-1$.

Now we are in position to state the final theorem in this section,
which follows from Theorem \ref{thm:unitaryintser}, Theorem \ref{thm:unitaryhimod} and the above argument.
\begin{theorem}
Any irreducible unitary Harish-Chandra module over $\g$ for the conjugate-linear anti-involution $\theta_{1,\be}^+$ is one of the following modules.\\
(1) $V(a,b,F)$ with $a\in\R^\times ,b\in\frac12+\sqrt{-1}\R$ and $\be_iF_{p-i,\ov{i+j}}=\ov{F_{i,j}}$ for any $j\in\cf,1\le i\le p-1$;\\
(2) a highest weight module of highest weight $\vf$ satisfying the conditions in Theorem \ref{thm:unitaryhimod};\\
(3) a lowest weight module of lowest weight $-\vf$, where $\vf$ satisfies the conditions in Theorem \ref{thm:unitaryhimod}.
\end{theorem}

\bigskip

\textbf{Acknowledgments:}
C. Xu is supported by the National Natural Science Foundation of China (No. 12261077).


\begin{thebibliography}{99}

\bibitem{Ba}
K. Barron.
On twisted modules for $N=2$ supersymmetric vertex operator superalgebras.
In: Lie Theory and its Applications in Physics. Springer Proceedings in Mathematics \& Statistics, vol. 36. Tokyo: Springer, 2013, 411-420.


\bibitem{CP1986}
V. Chari and A. Pressley.
New unitary representations of loop groups.
Math. Ann. 275(1986), 87-104.


\bibitem{CP1987}
V. Chari and A. Pressley.
A new family of irreducible, integrable modules for affine Lie algebras.
Math. Ann. 275(1987), 277-562.


\bibitem{CP1988}
V. Chari and A. Pressley.
Unitary representations of the Virasoro algebra and a conjecture of Kac.
Compos. Math. 67(1988), 315-342.


\bibitem{FF}
B. Feigin and D. Fuks.
Verma modules over the Virasoro algebra.
Funct. Anal. Appl. 17(1983), no. 3, 91-92.


\bibitem{KL}
V. Kac and J. Liberati.
Unitary quasi-finite representations of $W_\infty$.
Lett. Math. Phys. 53(2000), 11-27.


\bibitem{Kwon}
N. Kwon.
On the study of unitary representations of the twisted Heisenberg-Virasoro algebra via highest weight modules over affine Lie algebras.
J. Nonlinear Math. Phys. 21(2014), no. 4, 584-592.


\bibitem{LPXZ}
D. Liu, Y. Pei, L. Xia and K. Zhao.
Irreducible modules over the mirror Heisenberg-Virasoro algebra.
Commun. Contemp. Math. 24(2022), 2150026.


\bibitem{Ma}
O. Mathieu.
Classification of Harish-Chandra modules over the Virasoro Lie algebra.
Invent. Math. 107(1992), 225-234.


\bibitem{MP}
R. Moody and A. Pianzola.
Lie algebras with triangular decomposition.
CMS series of monographs and advanced texts. J. Wiley-Int. Publ. N.Y. 1995.


\bibitem{Palev1}
T. Palev.
Highest weight irreducible unitary representations of Lie algebras of infinite matrices. I. The algebra $\mathfrak{gl}(\infty)$.
J. Math. Phys. 31(1990), no. 3, 579-586.


\bibitem{Palev2}
T. Palev.
Highest weight irreducible unitarizable representations of Lie algebras of infinite matrices. The algebra $A_\infty$.
J. Math. Phys. 31(1990), no. 5, 1078-1084.


\bibitem{Xu}
C. Xu.
Classification of irreducible Harish-Chandra modules over gap-$p$ Virasoro algebras.
J. Algebra 542(2020), 1-34.


\bibitem{XCT}
C. Xu, F. Chen and S. Tan.
Non-weight modules over gap-$p$ Virasoro algebras.
arXiv:2505.15273.


\bibitem{ZT1}
X. Zhang and S. Tan.
$\theta$-unitary represnetations for the $W$-algebra $W(2,2)$.
Linear Multilinear Algebra, 60(2012), no. 5, 533-543.


\bibitem{ZT2}
X. Zhang and S. Tan.
Unitary representations for the twisted Schr{\"o}dinger-Virasoro algebra.
J. Algebra Appl. 12(2013), 1250132.


\bibitem{ZTL}
X. Zhang, S. Tan and H. Lian.
Unitary modules for the twisted Heisenberg-Virasoro algebra.
Algebra Colloq. 26(2019), no. 3, 529-540.


\end{thebibliography}
\end{document}